\documentclass[11pt, reqno]{amsart}

\usepackage{amsthm,amssymb,amstext,amscd,amsfonts,amsbsy,amsrefs,amsxtra,latexsym,amsmath,xcolor,mathrsfs,fancybox,upgreek, soul,url}
\usepackage[english]{babel}
\usepackage[all,cmtip]{xy}
\usepackage[utf8]{inputenc}
\usepackage[T1]{fontenc}
\usepackage{cancel}
\usepackage[draft]{hyperref}

\usepackage{comment}
\usepackage{mdframed}
\allowdisplaybreaks
\usepackage{mathtools}
\usepackage{enumerate}
\usepackage{thmtools}
\usepackage{thm-restate}
\usepackage{chngcntr}
\usepackage{enumitem}
\usepackage{etoolbox}
\usepackage{todonotes}
\usepackage{stmaryrd}
\usepackage{soul}

\usepackage{tikz}
\usepackage{verbatim}
\usetikzlibrary{quotes,angles}

\DeclarePairedDelimiter\abs{\lvert}{\rvert}
\DeclarePairedDelimiter\norm{\lVert}{\rVert}

\makeatletter
\let\oldabs\abs
\def\abs{\@ifstar{\oldabs}{\oldabs*}}
\let\oldnorm\norm
\def\norm{\@ifstar{\oldnorm}{\oldnorm*}}
\makeatother

\makeatletter
\glb@settings 
\makeatother

\usepackage[top=1.25in, bottom=1.25in, left=0.8in, right=0.8in]{geometry}

\newtheorem{theorem}{Theorem}
\newtheorem{lemma}[theorem]{Lemma}

\newtheorem{conjecture}[theorem]{Conjecture}

\theoremstyle{definition}

\theoremstyle{remark}

\numberwithin{theorem}{section}
\numberwithin{proposition}{section}
\numberwithin{lemma}{section}
\numberwithin{corollary}{section}
\numberwithin{equation}{section}
\numberwithin{conjecture}{section}

\setlist[enumerate,1]{before=}
\AfterEndEnvironment{enumerate}{}

\newcommand{\N}{\mathbb{N}}
\newcommand{\Z}{\mathbb{Z}}

\usepackage{bm}

\renewcommand{\bmod}[1]{\ ( \mathrm{mod} \, #1 )}

\title{Sums of proper divisors with missing digits}

\author{K\"{u}bra Benl\.{i}, Giulia Cesana, C\'{e}cile Dartyge, Charlotte Dombrowsky, \\ and Lola Thompson}

\address{Department of Mathematics and Computer Science, University of Lethbridge, 4401 University Drive, Lethbridge, Alberta, T1K 3M4 Canada}
\email{kubra.benli@uleth.ca}

\address{University of Cologne, Department of Mathematics and Computer Science, Division of Mathematics, Weyertal 86-90, 50931 Cologne, Germany}
\email{gcesana@math.uni-koeln.de}

\address{Institut \'{E}lie Cartan, Universit\'{e} de Lorraine, BP 70239, 54506 Vand\oe uvre-l\`{e}s-Nancy Cedex, France}
\email{cecile.dartyge@univ-lorraine.fr}

\address{Universiteit Leiden, Science Mathematisch Instituut, Mathematisch Instituut, Niels Bohrweg 1, 2333 CA Leiden, Netherlands}
\email{c.k.l.dombrowsky@math.leidenuniv.nl}

\address{Mathematics Institute, Utrecht University, Hans Freudenthalgebouw, Budapestlaan 6, 3584 CD Utrecht, Netherlands}
\email{l.thompson@uu.nl}

\begin{document}

\begin{abstract}
Let $s(n)$ denote the sum of proper divisors of an integer $n$. In 1992, Erd\H{o}s, Granville, Pomerance, and Spiro (EGPS) conjectured that if $\mathcal{A}$ is a set of integers with asymptotic density zero then $s^{-1}(\mathcal{A})$ also has asymptotic density zero. In this paper we show that the EGPS conjecture holds when $\mathcal{A}$ is taken to be a set of integers with missing digits. In particular, we give a sharp upper bound for the size of this preimage set. We also provide an overview of progress towards the EGPS conjecture and survey recent work on sets of integers with missing digits. 
\end{abstract}

\maketitle

\section{Introduction}

Let $s(n)$ denote the \textit{sum-of-proper-divisors function}, i.e., $s(n) = \sum_{d \mid n, d < n} d.$ The function $s(n)$ has been studied since the time of the ancient Greeks. Indeed, the \textit{perfect} numbers are those integers $n$ for which $s(n) = n$. The Greeks also spoke of integers as being \textit{deficient} if $s(n) <n$ and \textit{abundant} if $s(n) > n$.

%\textbf{KB: Need the definition of $\sigma(n)$ somewhere, as well. GC: I did this at the beginning of the proof of Theorem \ref{theorem: upper bound set of missing digits}, where $\sigma(n)$ first showed up. KB: I suggest that we give the definition before we start the proof. This could be done as a short subsection. Or we can mention it here as deficient and abundant numbers are easily defined in terms of $\sigma(n)$.}

It is natural to wonder how the function $s$ behaves when applied to various inputs. One surprising feature is that $s$ can map sets of asymptotic density\footnote{If $\mathcal{S}$ is a subset of the natural numbers, then the \textit{asymptotic density} of $\mathcal{S}$ is given by $$\lim_{x \rightarrow \infty} \frac{1}{x} \#\{n \leq x: n \in \mathcal{S}\},$$ provided that this limit exists.} zero to sets with positive asymptotic density. For example, consider the set $\mathcal{A}$ as the set of numbers $pq$, where $p$ and $q$ are distinct primes. 
%For example, if $\mathcal{A}$ is the set of numbers $pq$, where $p$ and $q$ are distinct primes, then $\mathcal{A}$ has asymptotic density zero but $s(\mathcal{A})$ has asymptotic density $1/2$. Indeed, 
Then $\#(\mathcal{A}\cap [1,x])\ll x(\log\log x)/\log x$ implies that $\mathcal{A}$ has asymptotic density zero. Moreover, Montgomery and Vaughan \cite{MontgomeryVaughan75} showed that the number of even integers less than $x$ which are not the sum of two primes is at most $ x^{1-\delta}$ for some $\delta>0$ and recently, Pintz \cite{Pintz} proved  that $\delta =0.28$ is an admissible explicit value. Since  $s(pq)=p+q+1$ for distinct primes $p$ and $q$, it follows that $s(\mathcal{A})$
contains almost all odd numbers, which allows one to conclude that $s(\mathcal{A})$ has asymptotic density $1/2$.

%{\bf KB: I removed the comments, and changed the order of sentences a bit here. There were some flaws in the argument. Feel free to fix what I wrote.}
%Let $E(X)$ be exceptional set for the Goldbach problem, that is  the set of the even integers less than $X$ which are not the sum of two primes. Montgomery and Vaughan \cite{MontgomeryVaughan75} proved that there exists $\delta >0$ such that $E(X)\ll X^{1-\delta}.$ Since $s(pq)=1+p+q$ for all primes $p,q$, we deduce that the set $s(\mathcal{A})$ contains almost all odd integers and is thus of asymptotic density $1/2$ whereas $|\{ pq\le x\}| \sim x(\log\log x)/\log x$.

%\textbf{CDa: I have inserted a small proof. We should perhaps give a reference for an explicit exponent $\delta$ in the result of Montgomery and Vaughan.
%KB: Thank you Cecile, I am checking papers citing this result for an explicit constant, but do we need that? I think depending on how we present this proof, we can give more details or choose to be concise. We can discuss this in the meeting. 
%CDa: you are right, we don't really need to have an explicit constant. CDo: I think the proof here is good. To be honest, there are still some steps that are not natural to me; but I guess that I could look them up in an introductory book and I don't think the purpouse of the paper should be to proof elementary statements.}

In addition to studying the images of various sets $\mathcal{A}$ under the function $s$, one can also consider what happens when taking $s^{-1}(\mathcal{A}).$ The preimages can also be a bit surprising. For example, there are positive density sets of integers whose preimages under $s$ have asymptotic density zero. In fact, Erd\H{o}s \cite{Erdos73} demonstrated  that there are sets $\mathcal{A}$ of positive asymptotic density for which $s^{-1}(\mathcal{A})$ is empty.

In 1992, Erd\H{o}s, Granville, Pomerance, and Spiro made the following conjecture that we will later refer to as the EGPS Conjecture.

\begin{conjecture}\cite[Conjecture 4]{egps}
%{\color{red}[Erd\H{o}s, Granville, Pomerance, Spiro, 1992, \cite{egps}]} 
%{\bf GC: I think it would be nicer to use the citation \cite[Conjecture 4]{egps}, since the year and the names are mentioned in the sentence before this conjecture already and since this would be more consistent. KB: I agree, I changed that keeping the previous version in percent sign. }
Let $\mathcal{A}$ be a set of integers with asymptotic density zero. Then $s^{-1}(\mathcal{A})$ also has asymptotic density zero. \end{conjecture}

The full EGPS Conjecture is still open today. However, it has been confirmed for certain specific sets $\mathcal{A}$. The following theorems give some examples of these sets. First, Pollack took $\mathcal{A}$ to be the set of primes.

%{\bf GC: I think here we should add an additional sentence. Something like ''The following theorems give some examples.``. CDo: I agree with Giulia.}

\begin{theorem}\cite[Theorem 1.11]{PollackArithmeticProperties} 
If $\mathcal{A}$ is the set of primes then $\# (s^{-1}(\mathcal{A})\cap [1,x]) = O\left(\frac{x}{\log x}\right)$, for all $x\geq 2$.
\end{theorem}

In the opposite direction, for sets of integers with `many' prime factors, Troupe proved the following theorem.

\begin{theorem}\cite[Theorem 1.3]{Troupeprimefactors} 
Let $\omega(m)$ denote the number of distinct prime factors of an integer $m$. For any fixed $\epsilon > 0$, if $$\mathcal{A} = \{m : |\omega(m) - \log\log m| > \epsilon \log\log m\}$$ then $s^{-1}(\mathcal{A})$ has asymptotic density zero. 
\end{theorem}

In other words, not only are numbers with a lot more than the ``normal'' number of prime factors rare, their preimages under $s$ are rare as well. 
This theorem implies that $\log\log n$ is the normal order of $\omega (s(n))$. Very recently, Pollack and Troupe \cite{PollackTroupeErdosKacsn} have
improved this, showing that $\omega (s(n))$ satisfies an Erd\H os--Kac-type theorem. This signifies that $\omega (s(n))$ asymptotically  has a normal distribution with mean and variance $\log\log n$.

There are several other sets $\mathcal{A}$ whose preimages under $s$ have been studied in recent years. For example, Pollack considered the case where $\mathcal{A}$ is a set of palindromes.

\begin{theorem}\cite[Theorem 1]{PollackPalindromes}\label{Theorem: PollackPalindromes}
If $\mathcal{A}$ is the set of palindromes in any given base, then $s^{-1}(\mathcal{A})$ has asymptotic density zero. 
\end{theorem}

%{\bf GC: I think we should add some text between the theorems.}

%\begin{theorem}\cite[\bf Add Theorem/Lemma]{PollackPomeranceThompson}  
%Let $\mathcal{A}$ be the set of squares up to $x\boldsymbol{\in ???}$. Then $s^{-1}(\mathcal{A})$ has asymptotic density zero. 
%\end{theorem}

%\textbf{CDo: I tried to find the precise theorem, but I could not find it. Are we sure that this is the right paper to cite? KB:  This is a corollary of Theorem 1.8 so I removed this and the comments. }

In another paper, Troupe took $\mathcal{A}$ to be the set of integers representable as a sum of two squares.

\begin{theorem}\cite[Theorem 1.2]{TroupeSumofsquares}
Let $\mathcal{A}$ be the set of integers $n \leq x$ that can be written as a sum of two squares. Then $\#(s^{-1}(\mathcal{A})\cap [1,x]) \asymp \frac{x}{\sqrt{\log x}}.$ 
\end{theorem}

Note that Troupe obtains both upper and lower bounds; most of the aforementioned papers only obtain upper bounds. In general, it is difficult to obtain nontrivial lower bounds. Another recent result, due to Pollack and 
 Singha Roy, shows that $k$-th power-free values of $n$ and $s(n)$ are equally common.

%%%Maybe explain why it is difficult to obtain nontrivial lower bounds. 

\begin{theorem}\cite[Theorem 1.3]{PollackRoyPowerfree}
Fix $k \geq 4$. On a set of integers with asymptotic density $1$, $$n \ \hbox{is $k$-free} \iff s(n) \ \hbox{is $k$-free}.$$ 
\end{theorem} 

The squarefree and cubefree cases remain conjectural. In a slightly different direction, there is a recent result by Lebowitz-Lockard, Pollack, and Singha Roy which shows that the values of $s(n)$ (for composite $n$) are equidistributed among the residue classes modulo $p$ for small primes $p$.

\begin{theorem}\cite[Theorem 1.3]{LebowitzLockardPollackRoy}
Fix $A > 0$. As $x \rightarrow \infty$, the number of composite $n \leq x$ with $s(n) \equiv a \bmod{p}$ is $(1 + o(1))x/p$, for every residue class $a \bmod{p}$ with $p \leq (\log x)^A.$ 
\end{theorem}

In the present paper, we consider the preimages under $s$ of a new set $\mathcal{A}$, namely, the set of integers with missing digits. We will elaborate more on these integers in Section 2. In particular, we will give some historical background and define the notation carefully. For now, we briefly discuss our main results. We show that the EGPS Conjecture holds for sets of integers with missing digits. Moreover, we prove a quantitative version of this result. 

\begin{theorem}\label{short theorem in intro} 
Fix $g\geq 3$, $\gamma\in(0,1)$, and a nonempty set $\mathcal{D}\subsetneq \{0,1,\dots, g-1\}$. For all sufficiently large $x$, the number of $n\leq x$ for which $s(n)$ has all of its digits in base $g$ restricted to digits in $\mathcal{D}$ is $O(x\exp(-(\log\log x)^\gamma).$ %{\color{magenta} where $c$ is a positive constant depending on %$g$ \textit{and $\gamma$}.} 
\end{theorem}
%{\color{magenta}K: A remark on why we have the condition that $0\in \mathcal{D}$. Also this should be added to the first line of the proof when recalling the setting. I believe this condition could be removed but I may be wrong as Cecile mentioned in the meeting on Monday. }

%%Lola: Maybe we should state our main theorem with 1+o(1) instead of gamma, as suggested by one of the referees?
%%CDo: I don't know, which version is better/clearer. But how it is stated now, either we need do delete the comment about c or intoduce the constant properly. 
Our main tool in the proof is an upper bound for the number of positive integers $n\leq x $ such 
$g^k\nmid \sigma (n)$ when $g^k$ is a large integer. Note that it was proved in \cite[Hauptsatz 2]{WatsonRamanujan} that the set of such integers has asymptotic density zero for fixed modulus $g^k$. By \cite[Theorem 2]{Pomerance77}, such a result can be made uniform in the modulus, see \cite[Lemma 2.1] {PollackArithmeticProperties}, stated as Lemma \ref{Pollacklemma} below. By sacrificing the uniformity, we obtain the following stronger upper bound.

\begin{lemma}\label{keylemma}
Let $g\geq 3$ be a given integer. Let $\gamma,\delta\in (0,1)$ and $A >0$ also be given. Then for integers 
$k \in \left[5 (\log_3 x ), A(\log_2 x)^\gamma\right]$, we have 
%%%$$\sum_{\substack{{n\le x}\\ {g^k\nmid \sigma (n)}}}1\ll x\exp \left(-c\left(\log\log x\right)^\delta\right).$$
$$\sum_{\substack{{n\le x}\\ {g^k\nmid \sigma (n)}}}1\ll x\exp \left(-\left(\log_2 x\right)^\delta\right), $$
where the constant implied by the $\ll$ notation depends on the choices of $g,A,\gamma,\delta$.
\end{lemma}
%{\color{magenta} CDa : in this Lemma I have followed a remark of referee 2  who said that we can take $k\in [(\log_3 x)^2, B(\log_2 x)^\gamma ]$ instead of  $k\in [A(\log_2 x)^\gamma, B(\log_2 x)^\gamma ]$. In fact, I think that we can even take $k\in [5 (\log_3 x), B(\log x)^\gamma ]$. The only constraint for the lower bound of $k$ is that we must have $\ell \le k$ with $\ell =\lfloor (1-\alpha)\frac{\log _3 x}{\log g}\rfloor$. Are you agree with this? K: We think 5 can be replaced by 1.}

As this lemma does not rely on any facts about integers with missing digits, it may also be useful in other contexts.

One may wonder how sharp the upper bound in Theorem \ref{short theorem in intro} really is. We note that the statement can be re-written in the following form: For all sufficiently large $x$, the number of $n \leq x$ for which $s(n)$ has all of its digits in base $g$ restricted to digits in $\mathcal{D}$ is at most $$\frac{x}{\exp((\log_2 x)^{1 + o(1)})}.$$ Since $s(p) =1$ for all primes $p$, then whenever the set $\mathcal{D}$ contains $1$, it follows that the size of the preimage set of $\mathcal{D}$ has $\pi(x)$, the prime counting function, as a lower bound. Since $\pi(x) \sim \frac{x}{\log x}$ as $x \rightarrow \infty$, we see that our exponent of $1+ o(1)$ is optimal for arbitrary $g$ and $\mathcal{D}$.

%%Lola: We need to make the restatement of this lemma in section 3 have the same numbering as in the introduction. I will try to figure out how to do this later. 

%%Lola: I think we can delete the following paragraph now:

%In Section \ref{section:missing digits}, we will introduce notation that will allow us to make the statement of our main result more precise. In particular, we will carefully define the term ``integers with missing digits'' to indicate which digits we are excluding. We give more technical variants of our main result in Theorem \ref{theorem: upper bound set of missing digits} and its corollary. 

%\textbf{(LT: I think that we should find a way to state our Theorem 2.1 here instead. That way, the reader will know what the main goal of our paper is! This will require introducing $\#\mathcal{W}_\mathcal{D}$ a bit sooner. I'm currently struggling to figure out how to do this without introducing a ton of notation in our introduction...)}

%({\bf KB: One way to do it would be having the corollary here defining $\mathcal{A}$ in words without mentioning the set of the digits considered. I am not sure, may be as following:

%Let $\mathcal{A}$ be a set of integers with missing digits base $g\geq 3$. Then $\#s^{-1}(\mathcal{A})$ has asymptotic density zero. 
%We can then mention that in the following sections, we make a more precise statement of the main result, including defining the term 'integers with missing digits'. 

The proofs of all of the aforementioned theorems make crucial use of the special forms that the numbers in these sets possess. The methods do not generalize to arbitrary sets with asymptotic density zero. In a different direction, the EGPS Conjecture has also been shown to hold for all sets of $n \leq x$ with cardinality up to about $x^{1/2}$. More precisely we have the following theorem by Pollack, Pomerance, and the last author, where the result is uniform in the choice of the set as long as its size is small. 

%{\bf GC: Can we give an easy reason why the methods do not generalize to arbitrary sets of density 0?}.

\begin{theorem}\cite[Theorem 1.2]{PollackPomeranceThompson}\label{Theorem: PollackPomeranceThompson}
Let $\epsilon =\epsilon(x)$ be a fixed function tending to $0$ as $x\to\infty$. If $\mathcal{A}\subset\mathbb{N}$ with   $\#\left(\mathcal{A} \cap [1,x]\right) \ll x^{1/2 + \epsilon(x)}$ then $s^{-1}(\mathcal{A})$ has asymptotic density zero. 
\end{theorem}

As a corollary, one can obtain (for example) that if $\mathcal{A}$ is the set of squares up to $x$ then $s^{-1}(\mathcal{A})$ has asymptotic density zero. Similarly, one can use Theorem \ref{Theorem: PollackPomeranceThompson} to prove that the preimage of a set of integers with `many' missing digits has density 0. For example, if we remove half of the possible digits then the size of the set of integers with missing digits will be about $O(x^{1/2}),$ and thus this is handled by Theorem \ref{Theorem: PollackPomeranceThompson}.
One can also deduce Theorem \ref{Theorem: PollackPalindromes} as a corollary of Theorem \ref{Theorem: PollackPomeranceThompson}, since the number of palindromes less than $x$ is $O(\sqrt{x})$.
%{\bf KB: Should't we have some condition on how many digits are missing in order to be able to use Theorem \ref{Theorem: PollackPomeranceThompson} for integers with missing digits? } 
%{\bf CD.: yes we can apply Theorem \ref{Theorem: PollackPomeranceThompson} when  $|{\mathcal D}|\le\sqrt{g}$. I will insert a remark after Corollary 2.2}

\vspace{0.2 in}

%{\bf Give a survey of recent results on checking the EGPS conjecture for specific sets (including the 2020 paper of Troupe, the two 2021 papers of Pollack/Roy, one of which also includes Leibowitz-Lockard as a co-author). GC: I think this comment is done and can go right? KB: Yes. I hide it here.}

\subsection*{Notation and conventions}
Throughout this paper, we will write $\#S$ to denote the number of elements in a set $S$. We will use $\sigma(n)$ to denote the \textit{sum-of-divisors function}, defined by $\sigma(n)\coloneqq \sum_{d\mid n} d$; $\varphi(n)$ to denote the \textit{Euler $\varphi$-function} for a positive integer $n$, defined by $\varphi(n)\coloneqq \#\{1\leq j \leq n: \gcd(n,j)=1\}$; $\operatorname{id}$ to denote the \textit{identity function}; and $\omega(n)$ to denote the \textit{number of distinct prime factors} of an integer $n$. We let $\mu(n)$ be the \textit{M\"obius function} defined as 
\begin{align*}
    \mu (n)\coloneqq\begin{cases}
        (-1)^r, &\text{ if}\, n=p_1\cdots p_r\, \text{with distinct primes } p_i,\\
        0, &\text{ if there exists a prime $p$ such that $p^2\mid n$.}
    \end{cases}
\end{align*}
For arithmetic functions $G$ and $H$, the \textit{convolution} $G\ast H$ is defined by 
$$G\ast H(n)\coloneqq \sum_{ab=n}G(a)H(b), $$
 for any positive integer $n$.

For two real functions $F$ and $G$ where $G$ is a nonnegative valued function, we say that $F(x)=O(G(x))$ if there exists a positive constant $C$ and a real number $x_0$ such that $|F(x)| \leq C |G(x)|$ for all $x \geq x_0$
and $F(x)=o(G(x))$ as $x\to a$ if 
$$\lim_{x\to a} \frac{F(x)}{G(x)}=0.$$
We will also at times use Vinogradov's notation $\ll$ as an alternative to Landau's Big $O$ notation. Namely, $F \ll G$ denotes that $F(x) = O(G(x))$. Similarly, $\gg$ is used to denote the parallel notion with the inequalities reversed in the Big $O$ definition.  We write $F \asymp G$ when there are positive constants $C_1$ and $C_2$ such that 
$C_1|F| < |G| < C_2 |F|$.
Furthermore we let $\log_k(x)$ denote the $k$-th iterate of the natural logarithm of $x$, e.g., $\log_3 x=\log\log\log x$. %{\bf GC: Do we want to use this notation then I will change it in the file.}

\section{Integers with missing digits}\label{section:missing digits}

%\textbf{Background on sets of missing digits, %including (but not limited to) the references %\cite{BanksShparlinski2004,Bourgain2015,Dartyge %Mauduit2000,ErdosMauduitSarkozy98,ErdosMauduitSarkozy99}.}\\

Let $g\in\N$, $g\geq 3$. We consider the base $g$ expansion of a positive integer $n$,
$$ n= \sum_{j\geq 0}\varepsilon_j(n)g^j,$$
%{\bf GC: What is this $N$ in the upper bound of %the sum? Later we have the sums up to infinity. %We should be consistent here. Since the %$\varepsilon_j(n)$ can be zero I would get rid of %the $N$ here as well. }
 with coefficients $\varepsilon_j(n)\in\{0,\dots,g-1\}$. Note that this sum is finite. For a proper subset $\mathcal{D}\subsetneq \{0,\dots,g-1\}$ such that $0\in\mathcal{D}$, and an arithmetic function $f$ taking positive integer values, we put 
 \begin{align}\label{def:WD}
 \mathcal{W}_{f,\mathcal{D}}\coloneqq \left\{ n :  f(n)= \sum_{j\geq 0}\varepsilon_j(f(n)) g^j, \varepsilon_j(f(n))\in\mathcal{D}\right\} 
 \end{align}
as the set of integers $n$ where the digits of $f(n)$ are restricted in the set $\mathcal{D}$.

In addition, for any $x\ge 1$ and any set of integers $A$ let $A(x)$ denote the set $A \cap [1,x]$. So for any $x\geq 2$, we put 
 $$\mathcal{W}_{f,\mathcal{D}}(x)\coloneqq \left( f^{-1}(\mathcal{W}_{\mathcal{D}}) \right)(x)=\left\{ n\leq x :  f(n)= \sum_{j\geq 0}\varepsilon_j(f(n)) g^j, \varepsilon_j(f(n))\in\mathcal{D}\right\}$$
 as a finite subset of  $\mathcal{W}_{f,\mathcal{D}}$.

In the particular case $f=\operatorname{id}$
%, where $\operatorname{id}$ is the identity function,
we  simply write $\mathcal{W}_{\mathcal{D}}$ (resp. $\mathcal{W}_{\mathcal{D}}(x)$) in place of $\mathcal{W}_{\operatorname{id},\mathcal{D}}$ (resp. $\mathcal{W}_{\operatorname{id},\mathcal{D}}(x)$).
 The elements of $\mathcal{W}_{\mathcal{D}}$ are frequently referred to as \textit{integers with missing digits} (or \textit{integers with restricted digits}).
 In this survey we will also use  the terminology proposed by  Mauduit, who referred to them as \textit{ellipsephic\footnote{The origin of this nomenclature comes from the fusion of the two Greek words ``ellipsis'' $=$ missing and ``psiphic'' $=$ digit.} integers}.
 %(cf. \cite[page 12]{Col2006} for a detailed explanation of the word ellipsephic) 
 %\textbf{CDo: Would someone mind if instead of giving a reference to the origin of the word elliphesic, I'd write a footnote with the explanation? GC: I would actually prefer this. Perhaps you can put the reference in the footnote tho. CDa: I am agree with both your two suggestions GC: I added it.}. 

Since  the set $\mathcal{D}$ is a proper subset of $ \{0,\dots,g-1\}$, and  since we have \begin{align}\label{number of elements of WD}
\#\mathcal{W}_{\mathcal{D}}\left(g^N-1\right)=\left|\mathcal{D}\right|^N,
\end{align}
the elements of $\mathcal{W}_{\mathcal{D}}$ form a sparse set, i.e.,
$$\lim_{N\rightarrow\infty} \frac{\#\mathcal{W}_{\mathcal{D}}\left(g^N\right)}{g^N}=0. $$
In other words, the set $\mathcal{W}_{\mathcal{D}}$ has asymptotic density zero.

% Since  the set $\mathcal{D}$ is a proper subset of $ \{0,\dots,g-1\}$, the elements of $\mathcal{W}_{\mathcal{D}}$ form a sparse set, i.e.
% \begin{align}\label{number of elements of WD}
%\#\mathcal{W}_{\mathcal{D}}\left(g^N-1\right)=\left|\mathcal{D}\right|^N.
%\end{align}
 %The set  $\mathcal{W}_{\mathcal{D}}$ has asymptotic density zero, which means that
%$$\lim_{N\rightarrow\infty} \frac{\#\mathcal{W}_{\mathcal{D}}\left(g^N\right)}{g^N}=0. $$

%{\bf KB: I made a change above as describing sparse set we also mean the set having asymptotic density zero. I keep the previous version hidden for a comparison. }

%\textbf{Background on sets of missing digits, including (but not limited to) the references %\cite{BanksShparlinski2004,Bourgain2015,Dartyge Mauduit2000,ErdosMauduitSarkozy98,ErdosMauduitSarkozy99}.}\\

When  $0\not\in\mathcal{D}$, we can adapt the definition \eqref{def:WD} by setting
$$\mathcal{W}_{f,\mathcal{D}}\coloneqq \left\{ n:  f(n)= \sum_{j= 0}^N\varepsilon_j(f(n)) g^j, \varepsilon_j(f(n))\in\mathcal{D},\ N\in\N\right\}.$$

In this section, as we survey the literature, we concentrate on the case with $0\in\mathcal{D}$ in order to avoid some complications in the statements of some results. However, the interested reader can find some results related to the sets $\mathcal{W}_{\mathcal{D}}$ with $0\not\in\mathcal{D}$ in \cite{Aloui15}. In our main theorem, we do not require $0\in\mathcal{D}$.

If we set $g=10$ and $\mathcal{D}=\{3,6,9\}$, then any number in $\mathcal{W}_{f,D}$ is divisible by $3$. However, if we exclude similar trivial obstructions, we can expect that the sequence of ellipsephic integers behaves like the sequence of the natural numbers.
A first question could be whether these integers are well-distributed in arithmetic progressions. 
Erd\H os, Mauduit, and S\'ark\"ozy \cite{ErdosMauduitSarkozy98} give an affirmative answer.
Their result is valid under the following two natural hypotheses for the set  $\mathcal{D}=\{d_1,d_2,\dots,d_t\}$, namely
\begin{align}\label{Hyp:D}
d_1=0\in\mathcal{D} \text{ and } \gcd(d_2,\dots,d_t)=1.
\end{align}
For $a,q\in\Z$ such that $\gcd(q,g(g-1))=1$, we introduce the set of the ellipsephic integers congruent to $a$ modulo $q$ by 
$$\mathcal{W}_{\mathcal{D}} (x,a,q) \coloneqq \left\{ n\in\mathcal{W}_{\mathcal{D}}(x) : n\equiv a\bmod{q}\right\}.$$
Erd\H os, Mauduit, and S\'ark\"ozy proved that if $\mathcal{D}$ satisfies \eqref{Hyp:D}, then there exist constants $c_1\coloneqq c_1(g,t)>0$ and $c_2\coloneqq c_2(g,t)>0$ such that 
\begin{align}\label{EMS}
\left|\#\mathcal{W}_{\mathcal{D}} (x,a,q)-\frac{\#\mathcal{W}_{\mathcal{D}} (x)}{q}\right|=O\left( 
\frac{\#\mathcal{W}_{\mathcal{D}} (x)}{q}\exp \left(-c_1 \frac{\log x}{\log q}\right)\right),
\end{align}
for all $a\in\Z$ and $q\le \exp (c_2\sqrt{\log x})$ such that $\gcd(q,g(g-1))=1$ (see \cite[Theorem 1]{ErdosMauduitSarkozy98}).
This result was improved by Konyagin \cite{Konyagin2001} in $2001$ and by Col \cite{Col2009} in $2009$.

The papers \cite{ErdosMauduitSarkozy98} and \cite{ErdosMauduitSarkozy99} provide several interesting applications
of these equidistribution results and give lists of open problems that inspired various research projects. 
%{\bf GC: Perhaps it would be nice to give some examples here.}

Another interesting aspect is the normal order of some arithmetic functions along ellipsephic integers.  
Banks and Shparlinski \cite{BanksShparlinski2004} obtained various results in this direction. In particular they studied 
 the average values of 
$\mathcal{W}_{\mathcal{D}}$ of  the Euler $\varphi$-function and the sum-of-divisors function.

Equation \eqref{EMS} can be seen as an analogue of the Siegel--Walfisz Theorem, stated below as Lemma \ref{SiegelWalfisz}, for primes in arithmetic progressions.
%{\bf Add reference?}
Such theorems are not applicable when the modulus $q$ is a power of $x$.
 However, in many applications, it is sufficient to use an equidistribution result that is averaged over the moduli $q$. In doing so, one is able to extend the range of $q$.
%\textcolor{blue}{However, in many applications, an equidistribution result averaged over the modulus $q$ is sufficient.}
%K: I liked this transition, so I removed the color.
One such application can be seen in work by the third author and Mauduit \cites{DartygeMauduit2000,DartygeMauduit2001}, and independently by Konyagin \cite{Konyagin2001}. They proved that there exists an $\alpha \coloneqq \alpha (g,\mathcal{D})$ such that 
\eqref{EMS} holds for almost all $q<x^{\alpha}$ satisfying $\gcd(q,g(g-1))=1$.
Such results combined with sieve methods imply the existence of  ellipsephic integers with few prime factors.
For example in \cite{DartygeMauduit2000} it is proved that there exist infinitely many $n\in\mathcal{W}_{\{ 0,1\}}$ with at most $k_g = (1+o(1))8g/\pi$ prime factors as $g\to \infty$.
%$k_g$ prime factors with $k_g = (1+o(1))8g/\pi$ (when the basis $g\rightarrow \infty$).

The problem of the existence of infinitely many primes with missing digits has been solved recently by Maynard \cites{Maynard19,Maynard22}.
In particular, he proved the following spectacular result, given by \cite[Theorem 1.1]{Maynard19}.
Let $a_0\in\{ 0,...,9\}$. The number of primes $p\le x$ with no digit $a_0$ in their base $10$ expansions is
$$\asymp \frac{x^{\frac{\log 9}{\log 10}}}{\log x}.$$
He also gives a condition to determine whether there are finitely or infinitely many $n$ such that $P(n)\in\mathcal{W}_{\mathcal{D}}$, for any given non-constant polynomial $P\in\Z [X]$, large enough base $g$, and $\mathcal{D}=\{ 0,\ldots,g-1\}\setminus\{a_0\}$ (see \cite[Theorem 1.2]{Maynard22}).
The papers \cite{Maynard19} and \cite{Maynard22} also provide deep results when the number of excluded digits is $\ge 2$.

 We end this short discussion on ellipsephic integers  
  by giving an incomplete list of recent references on this subject containing many other interesting results, namely  
 \cites{Aloui15,Biggs21,Biggs23,Col2009,KonyaginMauduitSarkozy}.

In the present paper, we are particularly interested in the preimages of the sets of ellipsephic integers under $s(n)$, namely $\mathcal{W}_{s,\mathcal{D}}(x)=\left(s^{-1}\left(\mathcal{W}_{\mathcal{D}}\right)\right)(x)$.  With this setup, our main result can be reformulated in the following manner:
	Let $g\geq 3$ be an integer and $\mathcal{D}\subsetneq \{0,\dots,g-1\}$ be a nonempty proper subset. Let $0<\gamma<1$ be given. Then for all sufficiently large $x$, %there exists a constant $c$ depending on $g$ and $\gamma$ %such that
 we have
 %%%$$ \#\mathcal{W}_{s,\mathcal{D}}(x) =\#s^{-1}\left(\mathcal{W}_{\operatorname{id},\mathcal{D}}(x)\right)\leq C_g \frac{x}{(\log x)^{1/g^{\log\log\log x}}}.$$
%%$$ \#\mathcal{W}_{s,\mathcal{D}}(x) =\#s^{-1}\left(\mathcal{W}_{\mathcal{D}}(x)\right)\ll x\exp (-c(\log\log x)^\gamma).$$
$$ \#\mathcal{W}_{s,\mathcal{D}}(x) =\#\left(s^{-1}\left(\mathcal{W}_{\mathcal{D}}\right)\right)(x)\ll x\exp (-(\log_2 x)^\gamma).$$
	%In particular, this upper bound does not  depend on %the choice of $\mathcal{D}$.

%{\bf CDa: our previous bound $x/(\log x)^{1/g^{\log\log\log x}}$ was a trivial bound because 
%$$(\log x)^{1/g^{\log\log\log x}} =\exp \left (\frac{\log\log x}{g^{\log\log\log x}}\right ) =\exp \left ( \frac{\log\log x}{(\log\log x)^{\log g}}\right )$$ 
%which is close to $1$ when $g\ge 3$ and  $x\rightarrow +\infty$. GC: I agree.}

%\noindent {\bf Remark.} If $|{\mathcal D}|\le \sqrt{g}$,
%then $|W_{{\mathcal D}}(x)|\ll x^{\frac{\log {|\mathcal{D}|}}{\log g}}\ll x^{1/2}.$
%So, in this case, the qualitative version of our main result follows from Theorem \ref{Theorem: PollackPomeranceThompson}. 
%{\bf GC: I don't see this immediately. This means we must have
%$$C_g \left (\frac{\log\log\log \left(g^N\right)}{\log\log \left(g^N\right)}\right )^{\frac{\log (g/|\mathcal{D}|)}{\log g}} \to 0$$
%as $N\to\infty$, right? Why is this true? KB: Yes. We have a positive power of a quantity approaching to 0 as $N\rightarrow\infty$.}

%%%%%I suggest we have the general form and the general proof in the paper. 

%%%%%\begin{proposition}
%%%%	If $\mathcal{A}$ is the set of integers missing the digit $g-1$ in base $g$, then $s^{-1}(\mathcal{A})$ has asymptotic density zero. In other words, for $\mathcal{D}\coloneqq\{0,\dots,g-2\}$, we have that
%% $$\lim_{N\rightarrow\infty} \frac{\# s^{-1}%%%(\mathcal{A}_{\mathcal{D}}\cap[1,g^N])}{g^N}= 0.$$
%\end{proposition}

\section{Preliminary Lemmata for Theorem \ref{short theorem in intro}}

We start with a well-known deep result on primes in arithmetic progressions.

\begin{lemma}[Siegel--Walfisz Theorem, \cite{TenenbaumITAN22}*{Theorem II.8.17, page 376}] \label{SiegelWalfisz}
%{\bf GC: I put the book in the references folder and watched it up. This Lemma is \cite{TenenbaumITAN22}*{Theorem 8.16} on page 375!! I think when we give an explicit number for the Theorem we don't need to add a page in addition. KB: Thanks Giulia. I was using an older edition of the same book. But, I think we should keep the page number in the book references. This is a long book. We can also use some other books as reference for this specific theorem, however you decide. Also a quick note, in Tenenbaum's book we also need the chapter included in theorem's numbers. Here, I am referring to Theorem II.8.17 on the next page (page 376) with the same error term stated here as it is the classical version (even though the one you mentioned is also enough for our purposes). GC: Perfect, thanks. You are right.}
For any constant $A>0$, and uniformly for $x\geq 3$, $1\leq q\leq (\log x)^A$, and $\gcd (a,q)=1$, we have 
$$ \sum_{\substack{p\leq x\\p\equiv a\bmod q}}\log p=\frac{x}{\varphi(q)} +O\left(x\exp\left(-c\sqrt{\log x}\right)\right),$$
where $c=c(A)$ is a strictly positive constant.
\end{lemma}
{\bf Remarks.} We will apply this lemma for some $q=g^\ell$ with $\ell\approx \log_3 x$. Norton \cite{Norton} and Pomerance \cite{Pomerance77} independently proved that there exists a constant $C>0$ such that for all $x\ge 3$, and for all integers $q$, $a$ with $\gcd(a,q)=1$, $q>0$ we have\footnote{The interested reader may find a more precise formulation in \cite[Section 6]{Norton}  and in \cite[Remark 1]{Pomerance77}.}
$$\Bigg |\sum_{\substack{{p\le x}\\ {p\equiv a \bmod q}}}\frac{1}{p}-\frac{\log_2 x}{\varphi (q)}\Bigg |\le C.$$
We could apply these results instead of the Siegel--Walfisz Theorem in the proof of Lemma \ref{keylemma}.
Another remark is when $q$ has the special form $q=g^c$ for some fixed $g$ (as is the case in our application), 
Elliott \cite{Elliott} and then Baker and Zhao \cite{BakerZhao} proved that
it is possible to have asymptotic estimates even when the size of $g^\ell$ is a power of $x$. Baker and Zhao proved that if $q=g^\ell$ with fixed $g$ 
then it is possible to obtain a result similar to Lemma \ref{SiegelWalfisz}
in the range $g^\ell \le x^{5/12-\varepsilon }$ with $\varepsilon >0$ arbitrarily small.

Recall that a \textit{multiplicative function} $f$ is a function that satisfies $f(uv)=f(u)f(v)$ whenever $\gcd(u,v)=1$. In particular, if $f$ is a multiplicative function which is not identically zero, then we have $f(1)=1$. Next, we  quote the following technical result for the average order of a multiplicative function. 
\begin{lemma}\label{Tenenbaumnormalorder}\cite[Corollary III.3.6, page 457]{TenenbaumITAN22} %{\bf GC: This is \cite[Corollary 3.6]{TenenbaumITAN22} on page 457!! Again I would get rid of the page.}
Let $\lambda_1,\lambda_2$ be constants such that $\lambda_1 >0$ and $0\le\lambda_2<2$.
For any multiplicative function $f$ such that 
$$0\le f(p^\nu)\le\lambda_1\lambda_2^{\nu-1}$$
for all primes $p$ and for $\nu=1,2,3, \dots,$ we uniformly have 
\begin{align}\label{boundsumf}
\sum_{n\le x}f(n)\ll x\prod_{p\le x}\left ( 1-\frac{1}{p}\right )\sum_{\nu\ge 0}
\frac{f(p^\nu)}{p^\nu}  
\end{align}
for all $x\geq 1$.
The implicit constant in \eqref{boundsumf} is less than
$$4(1+9\lambda_1+\lambda_1\lambda_2/(2-\lambda_2)^2).$$
\end{lemma}
%{\bf KB: Do we want to keep the bound for the implicit constant in the lemma above?}
%Using the above result, we obtain an upper bound for the number of positive integers $n\leq x $ such $g^k\nmid \sigma (n)$ when $g^k$ is a large integer. 
%This type of bounds were considered before, for example we have the following result stated in the following form in \cite[Lemma 2.1]{PollackArithmeticProperties} which follows from a more general result by Pomerance \cite[Theorem 2]{Pomerance77}. Note that, by sacrificing the uniformity, we will be able to obtain a better bound for the number of positive integers $n\leq x $ such $g^k\nmid \sigma (n)$ when $g^k$ is a large integer.
%Note that it was proved in \cite[Hauptsatz 2]{WatsonRamanujan} that the set of such integers has asymptotic density zero for fixed modulus $g^k$. As a result of \cite[Theorem 2]{Pomerance77}, one can recover uniformity in modulus as follows. 

\begin{lemma}\cite[Lemma 2.1]{PollackArithmeticProperties}\label{Pollacklemma}
Let $x\geq 3$. Let $q$ be a positive integer. Then 
$$\sum_{\substack{{n\le x}\\ {q \nmid \sigma (n)}}}1\ll \frac{x}{(\log x)^{1/\varphi(q)}},$$
    uniformly in $q$. 
\end{lemma}

Note that, by sacrificing the uniformity, we will be able to obtain a better bound for the number of positive integers $n\leq x $ such $g^k\nmid \sigma (n)$ when $g^k$ is a large integer.

To prepare for the proof of such a result we first prove the following observation. 
%{\bf KB: I should remark that all these lemmas and their proofs should be polished at some point.} 
\begin{lemma}\label{n=ab}
Let $m$ be a positive integer. Then every positive integer $n$ can be written uniquely as $n=ab$ with $\gcd(a,b)=1$ and 
\begin{align*}
    \mu^2(a)=1, \,\,\,p\mid a \text{ implies } p\equiv -1 \bmod {m} \quad \text{ and } \quad p\mid b \text{ implies } p^2\mid b \text{ or } p\not\equiv -1 \bmod {m}.
    \end{align*}
\end{lemma}
\begin{proof}
Let $n>1$ as the result is vacuously true when $n=1$.  Assume that $n$ has the following prime factorization: 
$$n=p_1^{e_1}p_2^{e_2}\dots p_{j}^{e_j},$$
 with $p_i\not= p_j$ if $i\not = j$. Changing the order if needed, without loss of generality assume that $p_1,\dots, p_J\equiv -1 \bmod{m}$ with $e_1=\dots = e_J=1$ and for $k=J+1,\dots, j$ either $p_k\not\equiv -1 \bmod{m}$ or $e_k>1$. Then, we choose 
$$a=p_1\dots p_J$$
and 
$$b=\frac{n}{a}=p_{J+1}^{e_{J+1}}\dots p_{j}^{e_j}.$$
This finishes the proof.
\end{proof}

{\bf Remark.} If we remove the condition that $\gcd(a,b)=1$ in Lemma \ref{n=ab}, then the decomposition $n=ab$ is not unique. For example, if $n$ has a prime divisor $q$ with $q\equiv -1\bmod m$ and $q^3\mid m$, then we can also choose 
$$a=q\prod_{\substack{{p\equiv {-1}\bmod m}\\ {p||n}}}
p \quad \text{ and }\quad  b=\frac{n}{a} .$$
 %We thank one of the referees for this observation.\\
We are now ready to prove our key lemma. 
%\begin{lemma}\label{bound for S_2}
%Let $g\geq 3$ be a given integer. Let $\gamma,\delta\in (0,1)$ and $A,B >0$ also be given. Then for integers 
%$k \in \left[A (\log_2 x )^\gamma, B(\log_2 x)^\gamma\right]$, we have 
%%%$$\sum_{\substack{{n\le x}\\ {g^k\nmid \sigma (n)}}}1\ll x\exp \left(-c\left(\log\log x\right)^\delta\right).$$
%$$\sum_{\substack{{n\le x}\\ {g^k\nmid \sigma (n)}}}1\ll x\exp \left(-\left(\log_2 x\right)^\delta\right), $$
%where the constant implied by the $\ll$ notation depends on the choices of $g,A,B,\gamma,\delta$.
%\end{lemma}

%{\bf Remark.} In this paper, we have presented a lemma sufficient for application to Theorem \ref{short theorem in intro}. In particular the base $g$ here is fixed. The idea of the proof could provide results valid for larger values of $g$ even if it is necessary to modify the range for the exponents $k$. 

\begin{proof}[Proof of Lemma \ref{keylemma}]
Let $\ell\le k$ be chosen later. By Lemma \ref{n=ab}, an integer $n$ can be written in a unique way as $n=ab$ with $\gcd(a,b)=1$, $\mu^2 (a)=1$ and $p\mid a$ implies $p\equiv -1\bmod{g^\ell}$ and $b$ such that 
$$p\mid b \Rightarrow p^2\mid b\ \text{or } p\not\equiv-1\bmod{g^\ell}.$$
Suppose that $n=ab$, written in the above form, is a positive integer such that $g^k\nmid\sigma (n)$. Then we claim that $n$ has at most $m\coloneqq \lfloor k/\ell\rfloor$ prime factors $p$ such that  $p\equiv -1\bmod{g^\ell}$  and $p^2\nmid n$.  Indeed, if $n$ has more than $k/\ell$ prime factors $p\equiv -1\bmod{g^\ell}$ with $p^2\nmid n$, say $p_1, p_2,\dots p_{m+1}$, then $\sigma(n)= (p_1+1)(p_2+1)\dots (p_{m+1}+1)K$, where $K$ is a positive integer. Since each $p_i \equiv -1\bmod{g^\ell}$, this implies $\sigma(n)= c_1g^\ell c_2g^\ell\dots c_{m+1}g^\ell K =c_1c_2\dots c_{m+1}Kg^{(m+1)\ell}$ with positive integers $c_i$. Since $(m+1)\ell\geq k$, this would imply that  $g^k\mid\sigma (n)$, which contradicts our assumption. Thus for such $n$, we have $\omega(a)\leq m$. Combining what we noted above, we obtain
$$\sum_{\substack{{n\le x}\\ {g^k\nmid \sigma (n)}}}1\le\sum_{\substack{{ab\le x}\\ {p\mid a\Rightarrow p\equiv -1\bmod{g^\ell}}\\ {\omega (a)\le m}\\
 {p\mid b\Rightarrow p^2\mid b \text{ or } p\not\equiv -1\bmod{g^\ell}}}}\mu^2 (a).$$

To find an upper bound for the sum on the right-hand side, we first deal with the condition $\omega (a)\le m.$ To do that, we use  Rankin's method, replacing $1$ with a nonnegative quantity that is at least $1$ when $\omega (a)\le m.$ Namely, let  $t\in (0,1)$  be a parameter so that
$t^{\omega (a)-m}>0$ for any positive integer $a$ and  $t^{\omega (a)-m}\geq 1$ if $\omega (a)\le m$. Then we get 
\begin{align*}
\sum_{\substack{{n\le x}\\ {g^k\nmid \sigma (n)}}}1&\le \sum_{\substack{{ab\le x}\\ {p\mid a\Rightarrow p\equiv -1\bmod{g^\ell}}\\ 
 {p\mid b\Rightarrow p^2\mid b \text{ or } p\not\equiv -1\bmod{g^\ell}}}}\mu^2 (a) t^{\omega (a)-m}\\
 &\le t^{-m}\sum_{\substack{{ab\le x}\\ {p\mid a\Rightarrow p\equiv -1\bmod{g^\ell}}\\ 
 {p\mid b\Rightarrow p^2\mid b \text{ or } p\not\equiv -1\bmod{g^\ell}}}}\mu^2 (a) t^{\omega (a)}.
 \end{align*}
%For arithmetic functions $g$ and $h$, the \textit{convolution function} $g\ast h$ is defined by 
%$$g*h(n)\coloneqq \sum_{ab=n}g(a)h(b), $$
%for any positive integer $n$.
We can rewrite the sum on the right in terms of multiplicative functions

$$\sum_{\substack{{ab\le x}\\ {p\mid a\Rightarrow p\equiv -1\bmod{g^\ell}}\\ 
 {p\mid b\Rightarrow p^2\mid b \text{ or } p\not\equiv -1\bmod{g^\ell}}}}\mu^2 (a) t^{\omega (a)} =\sum_{n\le x}\sum_{ab=n}G(a)H(b)=\sum_{n\le x} G\ast H(n),$$
 where 
 $G$ and $H$ are multiplicative functions defined by
 \begin{align*}
 G(p)\coloneqq\begin{cases} t, & \text{ if } p\equiv -1\bmod{g^\ell},\\ 
 0, & \text{ if } p\not\equiv -1\bmod{g^\ell},
 \end{cases}  %\text{ and } g(p^\nu)=0 \text{ for all }\nu\ge 2,
 \end{align*}
and  $G(p^\nu)\coloneqq0$ for all $\nu\ge 2$; and 
\begin{align*}
H(p)\coloneqq \begin{cases} 0, & \text{ if } p\equiv -1\bmod{g^\ell},\\
1, & \text{ if } p\not\equiv -1\bmod{g^\ell},
 \end{cases} %\text{ and }  h(p^\nu)=1 \text{ for all }\nu\ge 2.
\end{align*}
and  $H(p^\nu)\coloneqq 1$ for all $\nu\ge 2$.

Let $f=G\ast H$. The function $f$ is also multiplicative with 
$f(p)=G(p)+H(p)$, which gives
\begin{align*}
f(p)= \begin{cases} t, & \text{ if } p\equiv -1\bmod{g^\ell},\\
1, & \text{ otherwise.}
 \end{cases}
\end{align*}
For $\nu\ge 2$ we have
\begin{align*}
f\left(p^\nu\right) = \sum_{ab=p^\nu}G(a)H(b) =G(p)H(p^{\nu -1})+H\left(p^\nu\right)=\begin{cases} 1, &\text { if } \nu =2,\\
t+1, & \text { if } \nu\ge 3 \text{ and } p\equiv -1\bmod{g^\ell},\\
1, & \text{ if } \nu\ge 3 \text{ and } p\not\equiv -1\bmod{g^\ell}.
\end{cases}
\end{align*}
Now, we can apply  Lemma \ref{Tenenbaumnormalorder} with $\lambda _1 =2$ and $\lambda_2 =1$ to $f$ and obtain
\begin{small}
    \begin{align*}
\sum_{\substack{{n\le x}\\ {g^k\nmid \sigma (n)}}}1&\ll t^{-m}x\prod_{\substack{{p\le x}\\ {p\equiv -1\bmod{g^\ell}}}}\left( 1-\frac{1}{p}\right)\left( 1+\frac{t}{p}+\frac{1}{p^2}+(t+1)\sum_{\nu\ge 3}
\frac{1}{p^\nu}\right)\prod_{\substack{{p\le x}\\ {p\not\equiv -1\bmod{g^\ell}}}}\left( 1-\frac{1}{p}\right)
\sum_{\nu\ge 0}\frac{1}{p^\nu}.
\end{align*}
\end{small}
Next, we note that for any prime $p$, 
$$\left( 1-\frac{1}{p}\right)\sum_{\nu\ge 3}
\frac{1}{p^\nu}=\left( 1-\frac{1}{p}\right)\frac{1}{p^3}\sum_{\nu\ge 0}
\frac{1}{p^\nu}=\left( 1-\frac{1}{p}\right)\frac{1}{p^3}\frac{1}{(1-1/p)}=\frac{1}{p^3}$$
and
$$\left( 1-\frac{1}{p}\right)
\sum_{\nu\ge 0}\frac{1}{p^\nu} =\left( 1-\frac{1}{p}\right) \frac{1}{(1-1/p)}=1$$
which yield
\begin{align*}
\sum_{\substack{{n\le x}\\ {g^k\nmid \sigma (n)}}}1&\ll t^{-m}x\prod_{\substack{{p\le x}\\ {p\equiv -1\bmod{g^\ell}}}}\left( 1 +\frac{t-1}{p}+\frac{1-t}{p^2}+\frac{t}{p^3}\right).
\end{align*}
Furthermore, we see that
\begin{small}
\begin{align*}
 \prod_{\substack{{p\le x}\\ {p\equiv -1\bmod{g^\ell}}}}&\left( 1 +\frac{t-1}{p}+\frac{1-t}{p^2}+\frac{t}{p^3}\right)= \prod_{\substack{{p\le x}\\ {p\equiv -1\bmod{g^\ell}}}}\left( 1 +\frac{t-1}{p}\right)\left(1+\frac{1}{ 1+\frac{t-1}{p}}\frac{1-t}{p^2}+\frac{1}{ 1+\frac{t-1}{p}}\frac{t}{p^3}\right) %\\=&\prod_{\substack{{p\le x}\\ {p\equiv -1\bmod{g^\ell}}}}\left( 1 +\frac{t-1}{p}\right)\prod_{\substack{{p\le x}\\ {p\equiv -1\bmod{g^\ell}}}}\left(1+\frac{p}{ p+t-1}\frac{1-t}{p^2}+\frac{p}{p+t-1}\frac{t}{p^3}\right)
 \\=&\prod_{\substack{{p\le x}\\ {p\equiv -1\bmod{g^\ell}}}}\left( 1 +\frac{t-1}{p}\right)\prod_{\substack{{p\le x}\\ {p\equiv -1\bmod{g^\ell}}}}\left(1+\frac{1-t}{ p(p+t-1)}+\frac{t}{p^2(p+t-1)}\right) 
\end{align*}
\end{small}
where the second product is bounded by a constant. So we obtain
\begin{align*}
\sum_{\substack{{n\le x}\\ {g^k\nmid \sigma (n)}}}1&\ll t^{-m}x\prod_{\substack{{p\le x}\\ {p\equiv -1\bmod{g^\ell}}}}\left( 1 +\frac{t-1}{p}\right).
\end{align*}
In order to bound the product above, we  use the Taylor--Young formula for $\log(1-X)$ for $|X|\le 1/2$, and the Siegel--Walfisz Theorem, stated in Lemma \ref{SiegelWalfisz}. So, if $g^\ell$ is less than a power of $\log x$, we have uniformly for  $t\in (0,1)$,
\begin{align*}
\log \left( \prod_{\substack{{p\le x}\\ {p\equiv -1\bmod{g^\ell}}}}\left( 1 +\frac{t-1}{p}\right)\right)
=&\sum_{\substack{{p\le x}\\ {p\equiv -1\bmod{g^\ell}}}}\log \left( 1+\frac{t-1}{p}\right) \\
=&\sum_{\substack{{p\le x}\\ {p\equiv -1\bmod{g^\ell}}}}
\left (\frac{t-1}{p}+O\left (\frac{1}{p^2}\right )\right )\\
=& -(1-t)\sum_{\substack{{p\le x}\\ {p\equiv -1\bmod{g^\ell}}}} \frac{1}{p}
+O\left (\sum_{p\le x}\frac{1}{p^2}\right )\\
=& -\frac{(1-t)\log_2 x}{\varphi (g^\ell)}+O(1),
\end{align*}
where the last equality above is  deduced from the Siegel--Walfisz Theorem after partial summation, or more directly by the results of Norton  \cite{Norton} and Pomerance \cite{Pomerance77} mentioned just after Lemma \ref{SiegelWalfisz}.
Thus,  we obtain
$$\sum_{\substack{{n\le x}\\ {g^k\nmid \sigma (n)}}} 1\ll \frac{x}{t^{m}(\log x)^{(1-t)/\varphi (g^\ell)}}.$$
%for some absolute constant $\eta >0$.
It remains to choose $\ell$ and $t$. Recall that %$k \in \left[A (\log\log x )^\gamma, B(\log\log x)^\gamma\right]$
%$k \in \left[A (\log_2 x )^\gamma, B(\log_2 %x)^\gamma\right]$. 
$k \in \left[5 (\log_3 x ), A(\log_2 x)^\gamma\right]$. With our choices, we would like to have $t^{m}(\log x)^{(1-t)/\varphi (g^\ell)}\rightarrow \infty.$ 
%Let $\alpha,\alpha ',\beta \in\mathbb{R} $ such %that $0<\max (\gamma ,\alpha ')<\beta <\alpha %<1$ and $\delta < \alpha-\alpha '$. 

Let $\alpha,\alpha ' \in\mathbb{R} $ such that $0<\alpha ' <\alpha <1$ and $\gamma, \delta < \alpha-\alpha'$. We let %.\footnote{\textcolor{magenta}{CDa: Following a suggestion of referee 2, I have suppressed the $\beta$ but the previous version with the $\beta$ is still in LateX file and hidden with some \%}}
%\color{magenta}{For example, we can choose $\alpha =\frac{1+\max (\gamma ,\delta)}{2}$ and $\alpha' =\min (\frac{\gamma}{2}, \alpha -\delta-\frac{1-\delta}{4})$.}

%%%%$$\ell \coloneqq\left\lfloor(1-\alpha)\frac{\log \log\log x}{\log g}\right\rfloor,$$
$$\ell \coloneqq\left\lfloor(1-\alpha)\frac{\log_3 x}{\log g}\right\rfloor,$$
%%$$t\coloneqq 1-\frac{1}{(\log\log x)^{\alpha '}}$$
$$t\coloneqq 1-\frac{1}{(\log_2 x)^{\alpha '}}.$$
Note that we then have $\ell\le k$, $g^{\ell} \leq (\log_2 x)^{1-\alpha}$ and 
%%%$$m\le \frac{k}{\ell}\leq B \frac{(\log\log x)^\gamma}{(1-\alpha)\frac{\log \log\log x}{\log g}} \ll (\log\log x )^\beta.$$ 
$$ 1\le m\le \frac{k}{\ell}\leq A \frac{(\log_2 x)^\gamma}{\Big\lfloor(1-\alpha)\frac{\log_3 x}{\log g}\Big\rfloor} \leq 2A \frac{(\log_2 x)^\gamma}{(1-\alpha)\frac{\log_3 x}{\log g}} \le %(\log_2 x )^\beta
(\log_2 x)^\gamma$$
 for $x$ large enough.

 With these choices, we get 
%\begin{align*}
%t^{m}(\log x)^{(1-t)/\varphi (g^\ell)} \geq&\,\, t^{m}(\log x)^{(1-t)/g^\ell}\\\gg& \left(1-\frac{1}{(\log\log x)^{\alpha '}}\right)^{(\log\log x )^\beta}(\log x)^{\frac{1-%\frac{1}{(\log\log x)^{\alpha '}}}{(\log\log x)^{(1-\alpha)}}}
%\\ \gg & \exp\left((\log\log x )^\beta \log \left(1-\frac{1}{(\log\log x)^{\alpha '}}\right)+ \left(\frac{(\log\log x)^{\alpha '}-1}{(\log\log x)^{\alpha '+1-\alpha}} %\right)\log\log x\right)
%\\ \gg & \exp\left((\log\log x )^\beta \log \left(1-\frac{1}{(\log\log x)^{\alpha '}}\right)+ \left(\frac{(\log\log x)^{\alpha '}-1}{(\log\log x)^{\alpha '+1-\alpha}} \right) \log\log x\right).
%\end{align*}
%\begin{align*}
%t^{m}(\log x)^{(1-t)/\varphi (g^\ell)} \geq&\,\, t^{m}(\log x)^{(1-t)/g^\ell}\\\gg& \left(1-\frac{1}{(\log_2 x)^{\alpha '}}\right)^{(\log_2 x )^\beta}(\log x)^{\frac{1}{(\log_2 x)^{\alpha'+1-\alpha}}}
%\\ \gg & \exp\left((\log_2 x )^\beta \log \left(1-\frac{1}{(\log_2 x)^{\alpha '}}\right)+ \frac{\log_2 x}{(\log_2 x)^{\alpha '+1-\alpha}} \right)
%\\ \gg & \exp\left((\log_2 x )^\beta \log \left(1-\frac{1}{(\log_2 x)^{\alpha '}}\right)+ \frac{1}{(\log_2 x)^{\alpha '-\alpha}} \right).
%\\ \gg & \exp\left((\log\log x )^\beta \log \left(1-\frac{1}{(\log\log x)^{\alpha '}}\right)+ \left(\frac{(\log\log x)^{\alpha '}-1}{(\log\log x)^{\alpha '+1-\alpha}} \right) \log\log x\right).
%\end{align*}
\begin{align*}
t^{m}(\log x)^{(1-t)/\varphi (g^\ell)} \geq&\,\, t^{m}(\log x)^{(1-t)/g^\ell}\\\gg& \left(1-\frac{1}{(\log_2 x)^{\alpha '}}\right)^{(\log_2 x )^\gamma}(\log x)^{\frac{1}{(\log_2 x)^{\alpha'+1-\alpha}}}
\\ = & \exp\left((\log_2 x )^\gamma \log \left(1-\frac{1}{(\log_2 x)^{\alpha '}}\right)+ \frac{\log_2 x}{(\log_2 x)^{\alpha '+1-\alpha}} \right)
\\ =& \exp\left((\log_2 x )^\gamma \log \left(1-\frac{1}{(\log_2 x)^{\alpha '}}\right)+ \frac{1}{(\log_2 x)^{\alpha '-\alpha}} \right).
%\\ \gg & \exp\left((\log\log x )^\beta \log \left(1-\frac{1}{(\log\log x)^{\alpha '}}\right)+ \left(\frac{(\log\log x)^{\alpha '}-1}{(\log\log x)^{\alpha '+1-\alpha}} \right) \log\log x\right).
\end{align*}
%{\bf KB: I think we can talk about how much of these details we would like to be included in the final draft. I also added this to see if the argument follows through.  CDo: Is there a difference between the last 2 lines?}
%%%%%%%%%%
%%%%%%%%%%

Since $\gamma <\alpha-\alpha'$, the first term in the exponential above in absolute value is at most one half of the second.
% Hence  we find that for all $\varepsilon >0$ and $x\ge x_0$, %with $x_0=x_0(\varepsilon,\alpha,\alpha',\beta)$ we have
Hence for all $x\ge x_0$ with $x_0=x_0 (\alpha ,\alpha ',\gamma )$
we have
 %%$$\sum_{\substack{{n\le x}\\ {g^k\nmid \sigma (n)}}} 1 \ll x\exp \left( (1-\varepsilon)(\log\log x)^{\alpha -\alpha '}\right) \ll x\exp \left(-(\log\log x)^\delta\right),$$
 $$\sum_{\substack{{n\le x}\\ {g^k\nmid \sigma (n)}}} 1 \ll x\exp \left( -\frac{1}{2}(\log_2 x)^{\alpha -\alpha '}\right) \ll x\exp \left(-(\log_2 x)^\delta\right),$$
% This choice allows us to take 
 %$k=(\log\log x)^\beta (\log\log\log x)^{1-\alpha}$ for all $0<\beta <\alpha <1$, 
 where we have used the condition $\delta <\alpha -\alpha'$. This
finishes the proof.
\end{proof}

\section{Proof of Theorem \ref{short theorem in intro}}

Recall our basic setup: We have $g\in\N$, $g\geq 3$, $\mathcal{D} \subsetneq \{0,1,\dots,g-1\}$ nonempty, $x$ sufficiently large, and $0<\gamma<1$.  For a positive integer 
%$k \in \left[(\log\log x )^\gamma, A(\log\log x)^\gamma\right]$ 
$k \in \left[\frac{(\log_2 x )^\gamma}{\log (g/|\mathcal{D}|)}, 2\frac{(\log_2 x)^\gamma}{\log (g/|\mathcal{D}|)}\right]$,  we write
\begin{align*}
	\#\mathcal{W}_{s,\mathcal{D}}(x) = \sum_{\substack{n\in\mathcal{W}_{s,\mathcal{D}}(x) \\ \sigma(n)\equiv 0\bmod{g^k}}} 1 + \sum_{\substack{n\in\mathcal{W}_{s,\mathcal{D}} (x)\\ \sigma(n)\not\equiv 0\bmod{g^k}}} 1 \eqqcolon S_1 +S_2.
\end{align*}

We will now work on the  upper bounds for $S_1$ and $S_2$ separately. To start with, we use a rather weak bound for $S_2$, where we drop the condition on the digit restriction on $s(n)$. Then we apply Lemma  \ref{keylemma}  with $\gamma=\delta$ to obtain
\begin{align*}
	%%S_2 \leq \sum_{\substack{n\leq x\\ g^k \nmid \sigma(n)}} 1 = O\left( x\exp \left(-c\left(\log\log x\right)^\gamma\right)\right).
    S_2 \leq \sum_{\substack{n\leq x\\ g^k \nmid \sigma(n)}} 1 = O\left( x\exp \left(-\left(\log_2 x\right)^\gamma\right)\right).
\end{align*}

Next, we focus on finding an upper bound for $S_1$. Following a similar setting as in  \cite{PollackPalindromes}, for $s(n)=\sum_{j=0 }^{N}\varepsilon_j (s(n))g^j$, for some $N\geq 1$, we put
\begin{align*}
	B \coloneqq \sum_{j= 0}^{k-1}\varepsilon_j(s(n)) g^j 
\end{align*}
as the number formed by the $k$-rightmost digits of $s(n)$ such that $ s(n) \equiv B\bmod{g^k}$. Note that this implies $B\leq g^{k}-1.$
Let $n\in\mathcal{W}_{s,\mathcal{D}}(x)$ with $\sigma(n)\equiv 0\bmod{g^k}$. Then, since $g^k\mid \sigma(n)$, we have 
\begin{align*}
n= \sigma(n)-s(n) \equiv -s(n) \equiv -B\bmod{g^k}.
\end{align*}
So we can relax the condition  on $S_1$ with a congruence condition as follows
\begin{align*}
	S_1 = \sum_{\substack{n\in\mathcal{W}_{s,\mathcal{D}}(x) \\ \sigma(n)\equiv 0\bmod{g^k}}} 1  \leq \sum_{\substack{n\leq x\\ n\equiv -B\bmod{g^k} \\ B\in \mathcal{W}_{ \mathcal{D}}(g^{k}-1)}} 1 = \sum_{B\in \mathcal{W}_{ \mathcal{D}}(g^{k}-1)} \sum_{\substack{n\leq x\\ n\equiv -B\bmod{g^k}}} 1 \leq |\mathcal{D}|^k \Big\lfloor\frac{x}{g^k}\Big\rfloor+ |\mathcal{D}|^k,
\end{align*}
where in the last inequality we used, for $M$ a positive integer, $b\in\Z$ and $X\geq M$, that
$$\#\{1\leq a\leq X:\,a\equiv b\bmod{M} \}\leq \lfloor X/M\rfloor+1 $$
along with (\ref{number of elements of WD}).

Inserting our choice of $k$ yields
%%%$$S_1\ll x  \exp \left( -k \log \left(g/|\mathcal{D}|\right)\right) \ll x\exp \left(-\log \left(g/|\mathcal{D}|\right) (\log\log x)^\gamma\right).$$
$$S_1\ll x  \exp \left( -k \log \left(g/|\mathcal{D}|\right)\right) \ll x\exp \left(- (\log_2 x)^\gamma\right).$$

%{\bf GC: Shouldn't it be $x\exp (-\log (g/|\mathcal{D}|) (\log\log x)^\gamma)$ on the right side? The $\gamma$ should only be the exponent for the $\log\log(x)$, right? KB: Yes, fixed this.}\\

Thus, overall we obtain the following upper bound
%%%$$ \#\mathcal{W}_{s,\mathcal{D}}(x) =\#s^{-1}\left(\mathcal{W}_{\mathcal{D}}(x)\right)\ll x\exp (-c(\log\log x)^\gamma).$$
$$ \#\mathcal{W}_{s,\mathcal{D}}(x) =\#s^{-1}\left(\mathcal{W}_{\mathcal{D}}\right)(x)\ll x\exp (-(\log_2 x)^\gamma)$$    
as desired.\qed

\section{Some remarks on a lower bound on $ \#\mathcal{W}_{s,\mathcal{D}}(x)$} 

As indicated earlier, as soon as we have $1\in\mathcal{D}$, then the elements in $\mathcal{W}_{s,\mathcal{D}}$ are at least as frequent as the primes, since $s(p)=1$ for all primes $p$. In the case when ${\mathcal D}=\{ 0,\ldots ,g-1\}\setminus\{ a_0\}$ for some $a_0\in\{ 1,\ldots ,g-1\}$  and $g$ large enough, we can prove that $s(n)$ takes on infinitely many different values in $\mathcal{W}_{\mathcal{D}}$ by adapting the ideas used in \cite{Maynard22}.

As already remarked in the introduction, if $p$ and $q$ are two distinct primes then $s(pq)=p+q+1$.
It is thus sufficient to prove that a positive proportion of ellipsephic integers can be expressed as $1$ plus a sum of two primes. The arguments of \cite[Sections 6-8]{Maynard22}  imply that 
\begin{align}\label{sumof2primes}
\sum_{n\in \mathcal{W}_{\mathcal{D}}(g^N-1)}\sum_{p+q+1=n}\log p\log q = (c(\mathcal{D})+o(1))(g|\mathcal{D}|)^N,
\end{align}
with 

	\makeatletter
	\renewcommand{\maketag@@@}[1]{\hbox{\m@th\normalsize\normalfont#1}}
	\makeatother
 
\begin{footnotesize}
    \begin{align}\label{c(D)}
        c(\mathcal{D})\coloneqq\frac{g}{\varphi (g)^2}\#\left\{ (b_1,b_2)\in\{ 0,\ldots ,g-1\}^2 : \gcd(b_1b_2, g)=1 \text{ and } \exists \ d\in\mathcal{D} \text{ such that } b_1+b_2+1\equiv d\bmod{g}\right\} .
    \end{align}
\end{footnotesize}
When $g$ is large enough, $c(\mathcal{D} )>0$. Indeed, the set on the right-hand side  of \eqref{c(D)} is non-empty. It contains at least
$(b_1,b_2)=(1,g-1)$ if $a_0\not =1$, and   $(b_1,b_2)=(1,1)$ in the case $a_0=1$ .

The proof of \eqref{sumof2primes} consists of reproducing Sections 6-8 in \cite{Maynard22}, with one additional prime variable which we handle trivially. The only small difference is in the computation of the main term coming from the major arcs at the end of the proof of \cite[Lemma 7.2]{Maynard22}.

Let $r(n)$ denote the weight 
$$r(n)=\sum_{p_1+p_2+1=n}\log p_1\log p_2.$$

%Then for any odd integer $n$ with  $\sqrt{x}\le n\le x$,
We can apply an upper bound sieve to detect the primes $p_1$ such that $n-1-p_1$ is also prime.
By \cite[Theorem 3.11]{HalberstamRichert} and the remark that appears afterwards,
for any odd integer  $n\in [\sqrt{x}, x]$ we have that
\begin{align}\label{maj:r(n)}
    r(n)\ll (\log x)^2\#\{ p\le n,\ n-1-p=p'\} \ll x\prod_{\substack{{p> 2}\\ {p| n-1}}}\frac{p-1}{p-2}\le c_1x\log\log x,
\end{align}
%\begin{align}\label{maj:r(n)}
%r(n)\ll ( \log x )^2\sum_{\substack{{p_1\le x}\\ {p| (n-1-p_1)\Rightarrow p %>x^{1/100}}}}1\le c_1 x,
%\end{align}
for some absolute $c_1>0$. 
By \eqref{sumof2primes} and \eqref{maj:r(n)} we deduce that 
\begin{align}\label{preimages}
%\#\mathcal{W}_{s,\mathcal{D}}(x)\geq
\#\left\{n\in\mathcal{W}_{\mathcal{D}}(x): \ n-1 \text{ is a sum of two primes}\right\}
\geq \frac{c(\mathcal{D})-o(1)}{c_1\log\log x}\#\mathcal{W}_{\mathcal{D}}\left(x\right).
\end{align}
This implies that $s(n)$ takes infinitely many different values in 
$\mathcal{W}_{\mathcal{D}}$.
%{\bf CD: previously there was two inequalities in\eqref{preimages}. I have supressed the first part 
%\[
%\#\mathcal{W}_{s,\mathcal{D}}(x)\geq
%\left|\left\{n\in\mathcal{W}_{\mathcal{D}}(x): \ n-1 \text{ is sum of two primes}\right\} \right|
%,\]
%because we already know that if $1\in\mathcal{D}$,
%$\#\mathcal{W}_{s,\mathcal{D}}(x)\geq x/\geq \log x$ %which is a better lower bound than \eqref{preimages}.
%}
In this short section, we  wanted to to highlight a result which is a direct consequence of the proofs presented in \cite{Maynard22}. However, detecting ellipsephic primes is a considerably more difficult problem than detecting ellipsephic integers $n$ such that $n-1$ is a sum of two primes.
It is thus certainly possible to have a more precise lower bound than \eqref{preimages} for more general sets $\mathcal{W}_{\mathcal{D}}$  and with small base $g$. Also,  we can probably remove the $\log\log x$ in the denominator in \eqref{preimages} by combining the arguments in \cite[Section 3.2]{Vaughan}  with \cite{Maynard22}.

\section*{Acknowledgements}
We would like to thank the Women In Numbers Europe 4 conference organizers for providing us with the opportunity to work together. We would also like to thank the referees for their helpful comments, which greatly improved this paper. The first and third authors were supported by ANR grant ANR-20-CE91-0006  ArithRand. The first author was also supported by a PIMS postdoctoral fellowship at the University of Lethbridge. The last author was supported by the Max Planck Institute for Mathematics during the early stages of this project.

%\newpage

\end{document}